\documentclass[12pt]{article}
\usepackage{amssymb,amsmath, amsthm}
\usepackage{amssymb}
\usepackage{framed}
\usepackage{mathrsfs}
\usepackage{enumerate}
\usepackage{tikz}
\usetikzlibrary{matrix}
\usepackage[all]{xy}
\usetikzlibrary{shapes}
\usetikzlibrary{decorations.markings}
\usepackage{hyperref}
\usepackage{multirow}
\usepackage{bm}
\usepackage[algoruled,boxruled, linesnumbered, shortend, lined]{algorithm2e}
\usepackage{algpseudocode}
\usepackage{cite}
\usepackage{enumitem}
\setenumerate[0]{leftmargin=*}

\numberwithin{equation}{section}

\DeclareMathOperator*{\proj}{proj}
\DeclareMathOperator*{\dom}{dom\,}
\DeclareMathOperator*{\gph}{gph\,}

\newcommand{\inclu}[0] {\ar@{^{(}->}}

\newcommand{\cA}{\mathcal{A}}

\newcommand{\too}{\rightrightarrows}

\newcommand{\E}{{\R}^n}
\newcommand{\conv}{{\rm conv\,}}

\newcommand{\dist}{{\rm dist}}

\newcommand{\R}{{\bf R}}

\usepackage[margin=1.1in]{geometry}

\newcommand{\argmin}{\operatornamewithlimits{argmin}}
\newcommand{\LS}{\operatornamewithlimits{Limsup}}

\newcommand{\ls}{\operatornamewithlimits{limsup}}

\SetAlgoSkip{bigskip}
\SetKwInput{Input}{Input\hspace{34pt}{ }}
\SetKwInput{Output}{Output\hspace{27pt}{ }}
\SetKwInput{Initialize}{Initialize\hspace{2pt}{ }}
\SetKwInput{Parameters}{Parameters\hspace{2pt}{ }}

\newtheorem{theorem}{Theorem}[section]

\newtheorem{lemma}[theorem]{Lemma}

\newtheorem{corollary}[theorem]{Corollary}

\newtheorem{thm}{Theorem}[section]

\newtheorem{assumption}{Assumption}
\newtheorem{condition}{Conditions}

\usepackage{mathtools}
\DeclarePairedDelimiter{\dotp}{\langle}{\rangle}

\title{Conservative and semismooth derivatives are equivalent for semialgebraic maps}
\author{Damek Davis\thanks{School of ORIE, Cornell University,
		Ithaca, NY 14850, USA;
		\texttt{people.orie.cornell.edu/dsd95/}. Research of Davis supported by an Alfred P. Sloan research fellowship and NSF DMS award 2047637.}\qquad Dmitriy Drusvyatskiy\thanks{Department of Mathematics, U. Washington,
		Seattle, WA 98195; \texttt{www.math.washington.edu/{\raise.17ex\hbox{$\scriptstyle\sim$}}ddrusv}. Research of Drusvyatskiy was supported by the NSF DMS 1651851 and CCF 1740551 awards.}}

\begin{document}
\date{}
\maketitle

\abstract

Subgradient and Newton  algorithms for nonsmooth optimization require generalized derivatives to satisfy subtle approximation properties: conservativity for the former and semismoothness for the latter. Though these two properties originate in entirely different contexts, we show that in the semi-algebraic setting they are equivalent. Both properties for a generalized derivative simply require it to coincide with the standard directional derivative on the tangent spaces of some partition of the domain into smooth manifolds. An appealing  byproduct is a new short  proof that semi-algebraic maps are semismooth relative to the Clarke Jacobian.

\section{Introduction}
Algorithms for nonsmooth optimization, such as subgradient and semi-smooth Newton methods, crucially rely on first-order approximations of Lipschitz maps. Two seemingly disparate first-order approximation properties underlie existing results. 

\paragraph{Conservativity.}{
In full generality, Lipschitz continuous functions can be highly pathological, resulting in failure of subgradient methods to find any critical points \cite{daniilidis2020pathological}. 	 Consequently, it is essential to limit the class of functions under consideration.	
With this in mind, recent analysis of subgradient methods \cite{bolte2020conservative,davis2020stochastic,majewski2018analysis} requires \emph{conservativity} of a {generalized gradient mapping} $G(x)$ for $f$. This property, axiomatically explored by Bolte and Pauwels \cite{bolte2020conservative}, stipulates the validity of a formal chain rule
$$
\frac{d}{dt} (f \circ \gamma)(t) = \dotp{G(\gamma(t)), \dot \gamma(t)} \qquad \text{for a.e. $t \in (0, 1)$},
$$
along any absolutely continuous curve $\gamma$. 
Most importantly, conservativity holds automatically when $f$ is semialgebraic and $G$ is the {Clarke subdifferential} $\partial_C f$~\cite{davis2020stochastic}. More generally, Bolte and Pauwels~\cite{bolte2020conservative,bolte2020mathematical} showed that conservativity holds when $G$ is an output of an automatic differentiation scheme on a composition of semialgebraic functions.}

\paragraph{Semismoothness.}{
Newton methods for solving nonsmooth equations  $F(x)=0$, for a Lipschitz map $F$ require a \emph{semismoothness} property of $F$ with respect to a generalized Jacobian mapping $G(x)$. This property, introduced by Mifflin \cite{mifflin1977semismooth} and explored for Newton methods by Qi and Sun \cite{qi1993nonsmooth}, stipulates the estimate:
$$
F(y) - F(x) - G(y)(y - x) \subset o(\|x - y\|)\mathbb{B} \qquad \textrm{as } y \to x,$$
at any point $x \in \E$. Bolte, Daniilidis and Lewis~\cite{bolte2009tame} famously showed that semismoothness holds  when $F$ is semialgebraic and $G=\partial_C F$ is the Clarke Jacobian. }

\bigskip

Seemingly disparate, there is reason to believe conservativity and semismoothness are closely related. For example, Norkin's seminal work~\cite{norkin1980generalized,norkin1986stochastic} analyzed subgradient methods on functions $f$ that are semismooth with respect to certain generalized subdifferential operators. On the other hand, Ruszczy{\'n}ski~\cite{ruszczynski2020convergence} recently showed that semismoothness implies a weaker notion of conservativity along semismooth curves. 
In this paper, we prove that the notions of conservativity and semismoothness are equivalent for semialgebraic maps. 
Our proof strategy is to relate the two properties to an intermediate ``stratified derivative'' condition, already known from \cite{lewis2021structure} to be equivalent to conservativity. The end result applies to a natural family of (set-valued) directional derivatives $D(x, u)$, akin to $G(x) u$. In particular, $D(x, u)$ could be a directional derivative, the map induced by the coordinate-wise Clarke Jacobian, or the output of an automatic differentiation procedure.
\begin{theorem}[Informal]
Suppose $F(\cdot)$ and $D(\cdot,\cdot)$ are semialgebraic. Then conservativity and semismoothness  are both equivalent to the following condition: 
there exists a partition of $\R^d$ into finitely many semialgebraic $C^1$ manifolds such that for any point $x$ lying in a manifold $M$ equality holds:
$$D(x,u)=\{F'(x,u)\}\qquad \forall u\in T_M(x),$$ 
\end{theorem}

Thus both conservativity and semismoothness hold if and only if $D(x,u)$ coincides with the standard directional derivative $F'(\cdot,\cdot)$ on the tangent spaces of some partition of $\R^n$ into finitely many smooth manifolds. 
In dual terms, this means that semismooth generalized derivatives are simply selections of the map $(x,u)\to J(x)u$ where the rows of $J(x)$ are Clarke subgradients of $F_i$ shifted by the normal space. Building on \cite{bolte2020conservative}, the authors of \cite{lewis2021structure} recently showed that the latter property is equivalent to  $J(x)$ being conservative. Summarizing, the results of our paper together with those already available in the literature \cite{lewis2021structure,bolte2020conservative} imply that semismooth directional derivatives, stratified directional derivatives/subgradients, and conservative set-valued vector fields are  equivalent in the semi-algebraic setting.

\section{Notation and preliminaries}
In this section, we record basic notation and preliminaries from variational analysis and semi-algebraic geometry that will be used in the paper.

\subsection{Variational analysis}
We follow standard notation of variational analysis, as set out for example in the monographs \cite{RW98,Mord_1,kk_book,CLSW,penot_book,ioffe_book}.
Throughout,  fix a Euclidean $n$-dimension space, denoted by $\E$, equipped with an inner product $\langle \cdot,\cdot\rangle$ and the induced norm $\|x\|=\sqrt{\langle x,x\rangle}$. A set-valued map $G\colon\E\to\R^m$, maps points to subsets $G(x)\subset\R^m$. The domain and graph of $G$, respectively,  are defined by 
$$\dom G=\{x\in\E: G(x)\neq \emptyset\}\qquad \textrm{and}\qquad \gph G=\{(x,y)\in \E\times\R^m: y\in G(x)\}.$$
The set $\LS_{x\to \bar x} G(x)$ consists of all limits of sequences $y_i\in G(x_i)$ where $x_i$ is any sequence converging to $\bar x$.
The map $G$ is called outer-semicontinuous if $\gph G$ is a closed set. We say that $G$ is inner-semicontinuous on a set $Q\subset\E$ if for any  points $x\in Q$ and $y\in G(x)$ and for any sequence $x_i \xrightarrow[]{Q} x$, there exist points $y_i\in G(x_i)$ converging to $y$. 

The distance function and the projection onto a set $X\subset\E$ are defined by:
$$\dist(y,X):=\inf_{x\in X}~\|y-x\|, \qquad \proj(y,X):=\argmin_{x\in X}~\|y-x\|.$$
The deviation between sets $X,Y\subset\E$ will be measured by the Hausdorff distance
$$\dist(X,Y):=\max\left\{\sup_{x\in X}\dist(x,Y),\sup_{y\in Y}\dist(y,X)\right\}.$$
The directional derivative of any map $F\colon\E\to\R^m$ is defined by 
\begin{equation}\label{eqn:dir_deriv}
F'(x,u):=\lim_{t\searrow 0}~ \frac{F(x+tu)-F(x)}{t}.
\end{equation}
The map $F$ is called directionally differentiable if $F'(x,u)$ is well-defined, meaning the limit exists in \eqref{eqn:dir_deriv} for every $x, u\in \E$. It is straightforward to verify that when $F$ is locally Lipschitz continuous  and directionally differentiable, equality holds:
$$F'(x,u)=\lim_{t\searrow 0,~v\to u} \frac{F(x+tv)-F(x)}{t}$$
Abusing notation, for any curve $\gamma\colon [0,1)\to \R^n$, we define the one-sided velocity 
$\gamma'(0):=\lim_{t\searrow 0}\frac{\gamma(t)-\gamma(0)}{t}$ and we let $\dot{\gamma}(t)$ denote the derivative of $\gamma$ at any point $t\in (0,1)$ where $\gamma$ is differentiable.

Given any locally Lipschitz continuous map $F\colon\E\too\R^m$, the Clarke Jacobian of $F$ at $x$ is the set 
$$\partial_C F(x)=\conv\left\{\lim_{i\to\infty} \nabla F(x_i): x_i\xrightarrow[]{\Omega}  x\right\},$$
where $\Omega$ is the set of points at which $F$ is differentiable and $\conv(\cdot)$ denotes the convex hull.
Finally, for any $C^1$ manifold $M\subset\E$ and a point $x\in M$, the symbols $T_M(x)$ and $N_M(x)$ will denote the tangent and normal spaces to $M$ at $x$.

\subsection{Semialgebraic geometry}
We next collect a few elementary facts from semialgebraic geometry. For details we refer the reader to \cite{Coste-semi,Coste-min,Dries-Miller96}. All results in the paper hold more generally, and with identical proofs, for sets and functions definable in an o-minimal structure. We focus on the semialgebraic setting only for simplicity.

A set $Q\subset\E$ is called semialgebraic if it can be written as a union of finitely many sets defined by finitely many polynomial inequalities. A set-valued map $G\colon\E\too\R^m$ is called semi-algebraic if its graph is a semialgebraic set. Univariate semialgebraic functions are particularly simple.

\begin{lemma}[Curves]\label{lem:semi_smooth_uni}
For any semialgebraic map $\gamma\colon[0,1]\to \E$, there exists $\epsilon>0$ such that $\gamma$ is $C^1$-smooth on the open interval $(0,\epsilon)$. Moreover, as  long as the quotients $t^{-1}(\gamma(t)-\gamma(0))$ are bounded for all small $t>0$, the derivative $\gamma'(0)$ exists and equality $\gamma'(0)=\lim_{t\searrow 0} \dot{\gamma}(t)$ holds.\end{lemma}

Notice that the equality $\gamma'(0)=\lim_{t\searrow 0} \dot{\gamma}(t)$ can be interpreted as semi-smoothness of univariate semi-algebraic functions---an observation we will revisit.
An immediate consequence is that locally Lipschitz semi-algebraic maps  are  directionally differentiable. The following theorem moreover shows that semi-algebraic maps are ``generically smooth.'' To simplify notation, we use the term {\em $C^1$ semialgebraic partition of $\R^n$} to mean a partition of $\R^n$ into finitely many semialgebraic $C^1$  manifolds.

\begin{thm}[Generic smoothness]\label{thm:gen_smooth}
Let $F\colon\E\to\R^m$ be a semialgebraic map. Then there exists a $C^1$-semialgebraic partition $\cA$ of $\R^n$ such that the restriction of $F$ to each manifold $M\in \cA$ is $C^1$-smooth. 
\end{thm}

A useful property that blends variational analytic and semialgebraic constructions is the projection formula, proved in \cite[Proposition 4]{Lewis-Clarke}.

\begin{thm}[Projection formula]\label{thm:proj_formula}
Let $f\colon\E\to\R$ be a semialgebraic locally Lipschitz continuous function. Then there exists a $C^1$-semialgebraic partition $\cA$ of $\R^n$  such that for any point $x$ lying in a manifold $M\in \cA$ equality holds:
$$\{f'(x,u)\}=\langle \partial_C f(x),u \rangle \qquad \forall u\in T_M(x).$$
\end{thm}

The next result shows that any semialgebraic  map admits a semialgebraic single-valued selection.

\begin{lemma}[Semialgebraic selection]\label{lem:semialgebraic_select}
Any semi-algebraic  map $G\colon\E\too\R^m$ admits a semialgebraic selection $g\colon\dom G\to \R$ satisfying $g(x)\in G(x)$ for all $x\in \dom G$.
\end{lemma}

The following theorem, from \cite[Proposition 2.28]{dim}, shows that semialgebraic set-valued maps are generically inner-semicontinuous.
\begin{theorem}[Generic inner semicontinuity]\label{thm:inner_semicont}
Let $G\colon \E\rightrightarrows\R^m$ be a semialgebraic map. Then there exists a $C^1$-semialgebraic partition $\cA$ of $\R^n$  such that the restriction $G\mid_{\mathcal{M}}$ is inner-semicontinuous for every  $\mathcal{M}\in \cA$. 
\end{theorem}

 Many of the aforementioned results guarantee existence of certain partitions of $\R^n$. Consequently, it will be useful to refine partitions. A partition $\cA$ of $\R^n$ is called {\em compatible} with a collection $\mathcal{B}$ of sets in  $\R^n$  if each set in $\mathcal{B}$ is a union of some sets in $\mathcal{A}$.

\begin{theorem}[Compatible partitions]\label{thm:com_part}
Let $\mathcal{B}$ be a collection of finitely many semialgebraic sets. Then there exists a $C^1$-semialgebraic partition $\cA$ of $\R^n$  that is compatible with $\mathcal{B}$.
\end{theorem}

\section{Main results}
Throughout, we let $F\colon\R^n\to\R^m$ be a locally Lipschitz continuous and directionally differentiable map and let $D\colon\E\times \E\too\R^m$ be a set-valued map. The reader may view $D(x,u)$ as a ``generalized directional derivative'' of $F$ at $x$ in direction $u$. 

\begin{assumption}\label{ass:standing}
We introduce the following assumptions on $D$.
\begin{enumerate}
\item {\bf (Full domain)} The image $D(x,u)$ is nonempty for all $x,u\in\E$.
\item {\bf (Homogeneity)} The map $D$ is positively homogeneous in the second argument: 
$$D(x,0)=\{0\}\qquad \textrm{and}\qquad D(x,tu)=t D(x,u),$$
for all  $x,u\in \E$ and $t> 0$. 
\item {\bf (Lipschitz continuity)} The assignment $D(x,\cdot)$ is Lipschitz continuous locally uniformly in $x$. That is, for every point $\bar x\in \E$, there exists $L>0$ such that 
$$\dist(D(x,u_1),D(x,u_2))\leq L\|u_1-u_2\|,$$
for all  $u_1,u_2\in \E$ and all $x$ sufficiently close to $\bar x$.
\end{enumerate}
\end{assumption}

The first condition is self-explanatory. The second condition, which asserts that $\gph D(x,\cdot)$ is a cone, is natural for directional derivatives. The third condition is more nuanced but is again  mild. In particular, all three conditions hold for the directional derivative $D(x,u)=F'(x,u)$ and for any map of the form $D(x,u)=J(x)u$ where $J\colon\E\too\R^{m\times n}$ is locally bounded.

Clearly, $D(x,u)$ can be regarded as a generalized directional derivative of $F$ only if it accurately predicts the variations of $F$ at $x$ in direction $u$. There are a number of seemingly distinct conditions in the literature that model ``goodness of approximation,'' depending on context. We record the most relevant ones for us below and comment on each.

\begin{condition}
	We introduce the following conditions.
\begin{enumerate}
\item\label{eqn:NewtI} {\bf (Semismooth I)} For any point $x$ it holds: 
$$\LS_{y\to x}~\frac{F(y)-F(x)-D(y;y-x)}{\|y-x\|}=\{0\},$$
\item\label{eqn:NewtII} {\bf (Semismooth II)} For any point $x$ it holds: 
$$\LS_{y\to x}~\frac{F(y)-F(x)+D(y;x-y)}{\|y-x\|}=\{0\}.$$
\item\label{eqn:conserv} {\bf (Conservative derivative)} For any absolutely continuous curve $\gamma\colon [0,1]\to \R^n$, equality holds:
\begin{equation}\label{eqn:def_cons}
\left\{\frac{d}{dt}(F\circ \gamma)(t)\right\}=D(\gamma(t),\dot \gamma(t))\qquad \textrm{for a.e. } t\in (0,1).
\end{equation}
\item\label{eqn:strat_der} {\bf (Stratified derivative)} There exists a  semialgebraic $C^1$ partition $\cA$ of $\R^n$ such that equality 
$$D(x,u)=\{F'(x,u)\}$$
holds for any manifold $M\in \cA$,  $x\in M$, and any tangent vector $u\in T_M(x)$.
\item \label{it:clarke_strat} {\bf (Stratified subdifferential)} There exists a semialgebraic $C^1$ partition $\cA$ of $\R^n$ such that for any manifold $M\in \cA$ and  $x\in M$, it holds:
$$D(x,u)\subset J(x)u,$$
 where we set 
$$J(x)=\{A\in \R^{m\times n}: A_i\in \partial_C F_i(x)+N_{M}(x)\},$$
and $A_i$ denote the rows of $A$.
\end{enumerate}
\end{condition}

Let us discuss these conditions in turn. 

\paragraph{Semismoothness.}{
The first two conditions, \ref{eqn:NewtI} and \ref{eqn:NewtII},
play a central role for ensuring superlinear convergence of Newton-type algorithms for nonsmooth equations. We refer the reader the monograph \cite[Chapter 10]{kk_book} for details. To place these conditions in context, recall that any locally Lipschitz and directionally differentiable map $F$ satisfies the first-order approximation property \cite{shapiro1990concepts}:
\begin{equation}\label{eqn:b_der}
F(y)=F(x)+F'(x,y-x)+o(\|y-x\|)\qquad \textrm{as } y\to x.
\end{equation}
Condition \ref{eqn:NewtI} instead asserts $$F(y)- F(x)-D(y,y-x)\subset o(\|y-x\|){\bf B}\qquad \textrm{as } y\to x.$$  Notice that contrary to \eqref{eqn:b_der}, the value $D(y,y-x)$ is computed at the basepoint $y$. Condition~\ref{eqn:NewtII}  asserts instead $$F(x)- F(y)-D(y,x-y)\subset o(\|y-x\|){\bf B}\qquad \textrm{as } y\to x.$$   Clearly, when $D$ has the form $D(x,u)=J(x)u$ for some set-valued map $J\colon\E\to\R^{m\times n}$, conditions \ref{eqn:NewtI} and \ref{eqn:NewtII} are equivalent. In particular, if $F$ is directionally differentiable and $J(x)=\partial_C F(x)$ is the Clarke generalized Jacobian, both conditions reduce to the semismoothness property in the sense of \cite{mifflin1977semismooth,qi1993nonsmooth}. In the context of optimization, similar conditions have been used by Norkin \cite{norkin1980generalized,norkin1986stochastic} to establish asymptotic convergence guarantees of subgradient methods. }

\paragraph{Conservative derivative.}
Condition \ref{eqn:conserv} asserts that the generalized directional derivative $D$ satisfies a formal chain rule at almost every point along any absolutely continuous path $\gamma$. Equivalently, we may write condition \eqref{eqn:def_cons} as:
$$\{F'(\gamma(t),\dot\gamma(t))\}=D(\gamma(t),\dot{\gamma}(t))\qquad \textrm{for a.e. } t\in (0,1).$$
 This property is equivalent to the conservative derivatives $J(x)$ introduced in \cite{bolte2020conservative}, in the setting $D(x,u)=J(x)u$. Such generalized derivatives play a key role in justifying  automatic differentiation techniques for nonsmooth functions \cite{bolte2020mathematical}, as well as for analyzing the asymptotic behavior of  subgradient methods \cite{davis2020stochastic}.

\paragraph{Stratified derivative/subdifferential}{
Condition~\ref{eqn:strat_der} is geometrically intuitive. It simply asserts that there exists a partition of $\R^n$ into finitely many smooth manifolds, so that $D(x,\cdot)$ coincides with the directional derivative in directions tangent to the manifold containing the point $x$. 
Condition~\ref{it:clarke_strat} can be interpreted as a ``dual'' counterpart of \ref{eqn:strat_der}. Namely it stipulates that $D(x,u)$ is a selection of the map $(x,u)\mapsto J(x)u$, where the rows of $J(x)$ consist of the Clarke subdifferentials of the component functions $F_i$ shifted by a normal space $N_M(x)$. The ``duality'' between conditions~\ref{eqn:strat_der} and~\ref{it:clarke_strat} is explored in \cite{Lewis-Clarke}.}

\bigskip
\subsection{Equivalence of \ref{eqn:NewtI}-\ref{it:clarke_strat}}

The goal of our paper is to prove that if Assumption~\ref{ass:standing} holds and $F$ and $D$ are semialgebraic, then conditions  \ref{eqn:NewtI}-\ref{it:clarke_strat} are equivalent. The authors of  \cite{lewis2021structure}, building on \cite{bolte2020conservative}, recently proved the equivalence $\ref{eqn:conserv}\Leftrightarrow\ref{it:clarke_strat}$ for maps of the form $D(x,u)=J(x)u$ where $J(\cdot)$ locally bounded and outer-semicontinuous. Therefore we claim no originality with respect to the equivalence  $\ref{eqn:conserv}\Leftrightarrow \ref{it:clarke_strat}$. 

We begin by proving that  \ref{eqn:conserv} implies \ref{eqn:NewtI} and \ref{eqn:NewtII}. The following simple lemma will be useful.

\begin{lemma}\label{lemm:directional_equal}
Condition \ref{eqn:conserv} implies that for any absolutely continuous curve $\gamma\colon [0,1]\to \R^n$, equality holds:
\begin{equation*}
D(\gamma(t),\dot \gamma(t))=-D(\gamma(t),-\dot \gamma(t))\qquad \textrm{for a.e. } t\in (0,1). 
\end{equation*}
\end{lemma}
\begin{proof}
Fix an absolutely continuous curve $\gamma\colon [0,1]\to \R^n$ and define $\alpha(t)=\gamma(1-t)$. Then $\dot\alpha(t)=-\dot\gamma(1-t)$ and $(F\circ\alpha)'(t)=-(F\circ\gamma)'(1-t)$ for a.e. $t\in (0,1)$. Plugging in $\alpha$ in place of $\gamma$ in \eqref{eqn:def_cons} completes the proof.
\end{proof}

\begin{theorem}[Conservative derivatives and semismoothness]\label{thm:cons_newton}
Suppose that $F$ and $D$ are semialgebraic and that Assumption~\ref{ass:standing} holds. Then  condition \ref{eqn:conserv} implies both \ref{eqn:NewtI} and \ref{eqn:NewtII}.  Moreover, both implications are true if  \ref{eqn:conserv} only holds with respect to  semialgebraic curves $\gamma$.
\end{theorem}
\begin{proof}
Suppose condition \ref{eqn:conserv}  holds;
we aim to verify \ref{eqn:NewtI}. To this end, fix a point $x$ and assume without loss of generality $x=0$ and $F(x)=0$. 
We aim to prove
$$\LS_{d\to0}~\frac{F(d)-D(d,d)}{\|d\|}=\{0\}.$$
For any $t\in [0,1]$ define the function
$$\varphi(t):=\sup_{(d,v):~\|d\|= t,\, v\in D(d,d)} ~\|F(d)-v\|.$$
Condition~\ref{eqn:NewtI} will follow immediately once we establish $\varphi'(0)=0$. Clearly,  we may assume $\varphi(t)>0$ for all small $t>0$, since otherwise the equality $\varphi'(0)=0$ holds trivially. 
Lemma~\ref{lem:semialgebraic_select} yields semialgebraic curves $d(\cdot)$ and $v(\cdot)$ satisfying
$\|d(t)\|=t$, $v(t)\in D(d(t)$, $d(t))$ and $\|F(d(t))-v(t)\|\geq \frac{1}{2}\varphi(t)$ for all small $t>0$.
 Observe $d(0)=0$ and $\|d(t)\|/t=1$. Therefore Lemma~\ref{lem:semi_smooth_uni} shows that $d'(0)$ exists and   $\lim_{t \to 0}\dot{d}(t)=d'(0)$.
 
 Note the inclusion $v(t)/t\in D(d(t),\frac{d(t)}{t})$. Local Lipschitz continuity of $D(x,\cdot)$, implies that there exists $L>0$ and a semialgebraic curve $w(t)$ satisfying
 $w(t)\in D(d(t),\dot d(t))$ with 
 $$\left\|w(t)-\frac{v(t)}{t}\right\|\leq L\left\|\frac{d(t)}{t}-\dot d(t)\right\|$$
 for all small $t>0$.
Since the right side tends to zero as $t\searrow 0$, we  compute
\begin{align*}
\varphi'(0)=\lim_{t\searrow 0}\frac{\varphi(t)}{t}&\leq \ls_{t\searrow 0}\frac{2\|F(d(t))-v(t)\|}{t}=2\cdot \ls_{t\searrow 0}\left\|\frac{F(d(t))}{t}-w(t)\right\|.
\end{align*}
Lemma~\ref{lem:semi_smooth_uni} and  condition~\ref{eqn:conserv} guarantee 
\begin{align*}
\lim_{t\searrow 0} w(t)=\lim_{t\searrow 0}(F\circ d)'(t)=\lim_{t\searrow 0}\frac{F(d(t))}{t},
\end{align*}
Note that all the limits in the displayed equation exist due to the local monotonicity of semialgebraic functions. 
 We  conclude $\varphi'(0)=0$ and therefore condition~\ref{eqn:NewtI} holds. The proof of condition~\ref{eqn:NewtII} follows by identical reasoning, while taking into account Lemma~\ref{lemm:directional_equal}.
\end{proof}

Combining Theorem~\ref{thm:cons_newton} with Theorem~\ref{thm:proj_formula} immediately shows that semi-algebraic locally Lipschitz maps are semismooth with respect to the Clarke generalized Jacobian---the main result of \cite{bolte2009tame}. The proof just presented appears to be  simpler and more direct than the original argument in \cite{bolte2009tame}.

\begin{corollary}[Semi-algebraic maps are semismooth]
Any locally Lipschitz semi-algebraic map $F\colon\E\too\R^m$ satisfies condition~\ref{eqn:NewtI} for the map $D(x,u)=\partial_C F(x)u$.
\end{corollary}
\begin{proof}
 Note the inclusion $\partial_C F(x)\subset \{A\in \R^{m\times n}: A_i\in \partial_C F_i(x)\}$.
Let $\mathcal{A}_i$ be the partition ensured by Theorem~\ref{thm:proj_formula} for each coordinate function $F_i$. Using Theorem~\ref{thm:com_part}, let $\mathcal{A}$ be a semialgebraic partition that is compatible with each $\mathcal{A}_i$.
Fix a semialgebraic curve $\gamma$. Semialgebraicity implies that for any $M \in \cA$ and a.e.\ $t \in (0, 1)$ with $\gamma(t) \in M$, the inclusion $\dot\gamma(t) \in T_{M}(\gamma(t))$ holds. Consequently, condition \ref{eqn:conserv} holds with respect to all semialgebraic curves.
An application of Theorem~\ref{thm:cons_newton} completes the proof.
\end{proof}

The proof of the full equivalence between conditions~\ref{eqn:NewtI}-\ref{it:clarke_strat} will make use of the following simple linear algebraic fact.

\begin{lemma}\label{lem:line_alg}
For any sets $A,B\subset\E$ and a subspace $V\subset\E$, the equivalence holds:
$$A\subset B+V\qquad \Longleftrightarrow\qquad \proj(A,V^{\perp})\subset\proj(B,V^{\perp}).$$
\end{lemma}

We now have all the ingredients to prove the main the result of the paper.

\begin{theorem}[Equivalence]\label{thm:equiv_stat}
Suppose that $F$ and $D$ are semialgebraic and that Assumption~\ref{ass:standing} holds. Then conditions~\ref{eqn:NewtI}-\ref{it:clarke_strat} are equivalent.
\end{theorem}
\begin{proof}
We first establish the equivalences $\ref{eqn:NewtII}\Leftrightarrow\ref{eqn:conserv}\Leftrightarrow \ref{eqn:strat_der} \Leftrightarrow\ref{it:clarke_strat}$ by verifying the implications $\ref{eqn:conserv}\Rightarrow\ref{eqn:NewtII}\Rightarrow \ref{eqn:strat_der}\Rightarrow\ref{it:clarke_strat}\Rightarrow\ref{eqn:conserv}$. 

\smallskip
\noindent{\textrm{\em Implication $\ref{eqn:conserv}\Rightarrow  \ref{eqn:NewtII}$:}}
This was proved in Theorem~\ref{thm:cons_newton}.

\smallskip
\noindent{\textrm{\em Implication $\ref{eqn:NewtII}\Rightarrow \ref{eqn:strat_der}$}:}
Theorems~\ref{thm:gen_smooth} and \ref{thm:inner_semicont} yield a  partition $\cA$ of $\R^n$ into finitely many semialgebraic $C^1$ manifolds such that when restricted to  each manifold $M\in \cA$, the function $F$ is $C^1$-smooth and the map 
$$x\mapsto \gph D(x,\cdot)$$ is inner semicontinuous.
 Fix a manifold $M\in \cA$, a point $x\in M$, unit tangent vector $u\in T_M(x)$, and $v\in D(x,u)$.  Since $-u$ also lies in $T_M(x)$, we may find sequences $x_i \xrightarrow[]{M} x$  and $\tau_i\searrow 0$  with $\tau_i^{-1}(x_i-x)\to -u$. By inner-semicontinuity, there exist  
 sequences $(u_i,v_i)$ converging to $(u,v)$ and satisfying $v_i\in D(x_i,u_i)$. The assumed condition~\ref{eqn:NewtII} therefore guarantees
 $$\{0\}=\LS_{i\to 0}~ \frac{F(x_i)-F(x)+D(x_i,x-x_i)}{\|x_i-x\|}.$$ 
Rearranging and using linearity of $F'(x,\cdot)$ on $T_M(x)$ we deduce
 \begin{equation}\label{eqn:continue_dis}
 \{F'(x,u)\}=\{-F'(x,-u)\}=\LS_{i\to \infty}~D\left(x_i,\frac{x-x_i}{{\|x-x_i\|}}\right).
 \end{equation}
 Taking into account local Lipschitz continuity of  $D(y,\cdot)$ uniformly for all $y$ near $x$, we deduce that there exists $L>0$ and a sequence $z_i\in D\left(x_i,\frac{x-x_i}{{\|x-x_i\|}}\right)$  satisfying 
 $$\left\|z_i-\frac{v_i}{\|u_i\|}\right\|\leq L \left\|\frac{x-x_i}{\|x-x_i\|}-\frac{u_i}{\|u_i\|}\right\|,$$
 for all large indices $i$. Observe that the right side tends to zero. Continuing \eqref{eqn:continue_dis}, we deduce 
 $$ \{F'(x,u)\}=\lim_{i\to\infty} z_i=\lim_{i\to\infty} \frac{v_i}{\|u_i\|}=\{v\}.$$
We have thus proved $D(x,u)=\{F'(x,u)\}$ for all $u\in T_M(x)$, and therefore  \ref{eqn:strat_der} holds.

\smallskip
\noindent{\textrm{\em Implication $\ref{eqn:strat_der}\Rightarrow\ref{it:clarke_strat}$:}} Let $\cA$ be the partition of $\R^n$ into $C^1$ semialgebraic manifolds stipulated by condition \ref{eqn:strat_der}.
Using Theorem~\ref{thm:proj_formula} for each coordinate function $F_i$ and refining the partitions using Theorem~\ref{thm:com_part}, we arrive at a finite partition $\cA'$ of $\R^n$ into $C^1$ semialgebraic manifolds that is compatible  with $\cA$ and satisfies
$$\{F'_i(x,u)\}=\langle \partial_C F_i(x),u\rangle,$$
whenever $x$ lies in $M\in\cA'$ and $u\in T_M(x)$ is a tangent vector. Since for any such vector $u$, equality $\{F'(x,u)\}=D(x,u)$ holds by $\ref{eqn:strat_der}$, we conclude
$$D(x,u)=J(x)u\qquad \forall u\in T_M(x).$$
On the other hand, for any $u\notin T_M(x)$, we trivially have
$$J_i(x)u=\langle \partial_C F_i(x),u\rangle  +\langle N_M(x),u\rangle=\langle \partial_C F_i(x),u\rangle+\R=\R,$$
where $J_i(x)$ denotes the set of all $i$'th rows of matrices in $J(x)$. Therefore, in this case, the inclusion $D(x,u)\subset J(x)u$ holds trivially.

\smallskip
\noindent{\textrm{\em Implication $\ref{it:clarke_strat}\Rightarrow\ref{eqn:conserv}$:}}
Let $\cA$ be the partition of $\R^n$ into $C^1$ semialgebraic manifolds stipulated by condition \ref{it:clarke_strat}. Refining the partition using Theorems~\ref{thm:proj_formula} and~\ref{thm:com_part}, we may ensure that the equality holds:
$$\{F'(x,u)\}=J(x)u  \qquad \forall u\in T_M(x),$$
whenever $x$ lies in some manifold $M\in \cA$.

Consider now any absolutely continuous curve $\gamma\colon [0,1]\to\E$. It is elementary to verify  that for a.e. $t\in (0,1)$ the implication holds (e.g. \cite[Lemma 4.13]{DIL}):
\begin{equation}\label{eqn:implication_tan}
\gamma(t)\in M\quad \Longrightarrow \quad \dot{\gamma}(t)\in T_{M}(\gamma(t)).
\end{equation}
We conclude that for  a.e. $t\in (0,1)$ and $M\in\cA$ satisfying $\gamma(t)\in M$, we have
$$\left\{\frac{d}{dt}(F\circ \gamma)(t)\right\}=\{F'(\gamma(t),\dot{\gamma}(t))\}=J(\gamma(t))\dot\gamma(t)\supset D(\gamma(t),\dot{\gamma}(t)),$$
as claimed.

Summarizing, we have proved the equivalence $\ref{eqn:NewtII}\Leftrightarrow\ref{eqn:conserv}\Leftrightarrow \ref{eqn:strat_der}\Leftrightarrow\ref{it:clarke_strat}$. Next,  observe that \ref{it:clarke_strat} holds for a map $D(\cdot,\cdot)$ if and only if it holds for the map $\hat D(x,u):=-D(x,-u)$. Noting that 
 condition \ref{eqn:NewtII} for $\hat D$ coincides with condition \ref{eqn:NewtI}   for $D$ completes the proof of the theorem.
\end{proof}

\bibliographystyle{plain}
\bibliography{bibliography}

\end{document}